\documentclass[12pt]{amsart}

\usepackage{newtxtext,newtxmath} % better replacement for times

\usepackage{amsmath,amssymb,amsfonts,amsthm}
\renewcommand{\Re}{\operatorname{Re}}

\usepackage{graphicx}
\usepackage{float}
\usepackage{booktabs}
\usepackage{multirow}

\usepackage{subcaption}

\usepackage{tikz}
\usetikzlibrary{patterns,positioning}
\usepackage{pict2e}

\usepackage{bm,bbm}
\usepackage{verbatim}
\usepackage{ragged2e}
\usepackage{enumitem}
\usepackage[colorinlistoftodos]{todonotes}

\usepackage[colorlinks=true,citecolor=blue,linkcolor=blue,urlcolor=blue]{hyperref}

\newlist{Henum}{enumerate}{1}
\setlist[Henum,1]{label=\textbf{H\arabic*}, ref=H\arabic*}

\newtheorem{thm}{Theorem}
\newtheorem{lem}[thm]{Lemma}
\newtheorem{remark}{Remark}


\begin{document}

\title[OT Stability for Inverse Point Sources]{Optimal-Transport Stability of Inverse Point-Source Problems for Elliptic and Parabolic Equations}

% (Optional) short title for headers
% \title[OT Stability for Inverse Point Sources]{Optimal-Transport Stability of Inverse Point-Source Problems for Elliptic and Parabolic Equations}

% =========================
% Authors (fill if needed)
% =========================
\author{Lingyun Qiu}
\address{Yau Mathematical Sciences Center, Tsinghua University, Beijing,  China.
Yanqi Lake Beijing Institute of Mathematical Sciences and Applications, Beijing,  China.}
\email{lyqiu@tsinghua.edu.cn}
\author{Shenwen Yu}
\address{Department of Mathematical Sciences, Tsinghua University, Beijing,  China}
\email{ysw22@mails.tsinghua.edu.cn}

% =========================
% Abstract
% =========================
\begin{abstract}
We establish quantitative global stability estimates, formulated in terms of optimal transport (OT) cost,
for inverse point-source problems governed by elliptic and parabolic equations with spatially varying coefficients.
The key idea is that the Kantorovich dual potential can be represented as a boundary functional of suitable adjoint solutions,
thereby linking OT geometry with boundary observations.
In the elliptic case, we construct complex geometric optics solutions that enforce prescribed pointwise constraints,
whereas in the parabolic case we employ controllable adjoint solutions that transfer interior information to the boundary.
Under mild regularity and separation assumptions, we obtain estimates of the form
\[
\mathcal{T}_c(\mu,\nu) \le C\,\|u_1 - u_2\|_{L^2(\partial\Omega)}
\quad \text{and} \quad
\mathcal{T}_c(\mu,\nu) \le C\,\|u_1 - u_2\|_{L^2(\partial\Omega\times[0,T])},
\]
where $\mu$ and $\nu$ are admissible point-source measures.
These results provide a unified analytical framework connecting inverse source problems and optimal transport,
and establish OT-based stability theory for inverse source problems governed by partial differential equations
with spatially varying coefficients.
\end{abstract}

\maketitle

% =========================
% Body starts here
% =========================
% \section{Introduction}
% ...

\section{Introduction}

In recent years, increasing attention has been devoted to environmental pollution, particularly to the transport of contaminants in aquatic and atmospheric media. In this work, we consider the following parabolic equation to model the spatiotemporal distribution of a pollutant:
\begin{equation*}
    \partial_t u(x,t) - \nabla \cdot\Big(\kappa(x) \nabla u(x,t)\Big) + q(x) u(x,t)= F(x,t),
\end{equation*}

where $\kappa(x)$ and $q(x)$ denote the spatially dependent diffusion and reaction coefficients, respectively, and $F$ represents the (possibly) time-dependent source term associated with the pollutant. The primary objective of this study is to establish quantitative stability for the unknown source term $F$ from boundary measurements, leading to an inverse source problem for parabolic equations. As a preliminary step toward the parabolic case, we first examine the associated stationary problem
\begin{equation*}
     - \nabla \cdot\Big(\kappa(x) \nabla u(x)\Big) + q(x) u(x)= F(x),
\end{equation*}
where several analytical ingredients and auxiliary propositions developed for the elliptic framework will subsequently be utilized in the parabolic analysis.

For general source terms, uniqueness for both inverse source problems fails. To ensure identifiability while retaining physical relevance, we restrict attention to source terms represented as finite superpositions of point sources. This point-source model has been extensively studied in the inverse problems literature. 
A seminal contribution was made by A.~El Badia and collaborators~\cite{ElBadiaHa-Duong2000}, who developed an algebraic reconstruction method based on harmonic polynomials $v(x_1,x_2)=(x_1+ix_2)^n$. These functions solve the adjoint equations with constant coefficients, and evaluation of the associated boundary functionals $\mathcal{R}(v)$ yields a linear system from which both source locations and amplitudes can be recovered.
Following this idea, inverse point-source reconstruction has been investigated for a variety of governing equations, including Helmholtz equations~\cite{badia_inverse_2011}, parabolic equations~\cite{ElBadiaHa-Duong2002, andrle_inverse_2015}, and wave equations~\cite{ohe_real-time_2011}, as well as moving-source scenarios~\cite{oheRealtimeReconstructionMoving2020, andrle_identification_2012}. 
The recent work~\cite{qiu_traceability_2025} introduced explicit time-dependent complex geometric optics (CGO) solutions $v(x,t)=e^{\alpha t}e^{\rho\cdot x}$ as test functions, yielding a well-conditioned reconstruction operator and avoiding the unstable rank estimation inherent in the harmonic-polynomial approach. 
Alternative methodologies have also been developed, ranging from direct sampling~\cite{qiu_direct_2024} and boundary integral formulations~\cite{sun_novel_2025} to machine-learning-based approaches~\cite{hu_point_2024} and optimization via data-misfit minimization~\cite{huang_point_2025}, all targeting the reconstruction of inverse point-source models.
While the works mentioned above provide valuable insights, they primarily address uniqueness and reconstruction algorithms; quantitative stability results remain limited. Existing approaches typically rely on explicit test functions available only in constant-coefficient settings, and thus encounter substantial obstacles when extended to partial differential equations (PDEs) with spatially varying parameters.

In this paper we address this more general and technically demanding setting, where $\kappa(x)$ and $q(x)$ are spatially dependent, and we establish global stability estimates expressed in terms of optimal transport (OT) cost. More precisely, given two admissible sources $\mu$ and $\nu$ (elliptic) or two admissible space-time sources (parabolic), we measure their discrepancy via an optimal-transport cost $\mathcal{T}_c(\mu,\nu)$ for a bounded cost function $c(x,y)$. 
Our goal is to bound $\mathcal{T}_c(\mu,\nu)$ by the discrepancy of the corresponding boundary data:
\[
\begin{aligned}
\mathcal{T}_c(\mu,\nu) &\le C\,\|u_1-u_2\|_{L^2(\partial\Omega)}
\quad &&\text{(elliptic)},\\[0.5ex]
\mathcal{T}_c(\mu,\nu) &\le C\,\|u_1-u_2\|_{L^2(\partial\Omega\times[0,T])}
\quad &&\text{(parabolic)}.
\end{aligned}
\]
The mechanism behind these estimates is that the Kantorovich dual potentials can be realized as boundary linear functionals of suitably constructed adjoint solutions. 
In the elliptic case, we employ CGO-based interpolation functions that enforce prescribed pointwise conditions at source locations. In the parabolic case, we design controllable adjoint solutions that transport interior information to the boundary in space-time. These constructions link OT duality with the boundary observation operator and provide a unified analytic route to stability for both stationary and time-dependent problems.

Over the past two decades, OT theory has undergone significant development; see the monographs~\cite{villani_optimal_2009,villani2003topics,alessio_figalli_invitation_2023}. 
Within inverse problems, there has been increasing interest in replacing traditional $L^2$-based misfit with Wasserstein distances in least-squares formulations. As shown in~\cite{engquist_application_2014,engquist_optimal_2022}, OT-based metrics can offer improved robustness and favorable convexity properties with respect to translations and dilations in seismic full waveform inversion.
From a numerical and analytical standpoint, \cite{qiu_full-waveform_2017,qiu_analysis_2021} provide systematic analysis and efficient algorithms for gradient-based methods employing Wasserstein misfits. In \cite{liu_moments_2025}, the authors employ the OT metric to derive error estimates for numerical solutions to inverse initial value problems for parabolic equations. To the best of our knowledge, however, rigorous OT-based stability results for inverse problems have not been available; one of the aims of this work is to fill this gap.

To accommodate spatially varying coefficients, correction terms arise in the construction of CGO solutions, and suitable estimates for these terms become essential. CGO solutions form a fundamental class of test functions in inverse problems: the celebrated work~\cite{sylvester_global_1987} established global uniqueness for the Calder\'on problem; subsequent developments extended CGO constructions to Schr\"odinger-type equations with advection terms~\cite{nakamura_global_1995,salo2004inverse} and to solutions vanishing on parts of the boundary~\cite{kenig_calderon_2007,kenig_calderon_2013}. Our analysis relies on $C^k$-norm estimates for the correction term as in~\cite{bal_inverse_2010}, where such bounds were used to obtain stability in photoacoustic tomography; see also~\cite{qiu_lipschitz_2023,carstea_calderon_2021,feizmohammadi_inverse_2020}. 
Related results for the variable-coefficient Helmholtz equation, obtained under suitable smallness assumptions, can be found in~\cite{bao_recovering_2021}. 
Complementing these CGO-based constructions, a second key ingredient in our analysis is boundary control theory: in the parabolic case, we use boundary controllability to transfer interior information to the boundary and to convert the relevant interior integrals into boundary terms, with quantitative bounds provided by the observability estimates in~\cite{fursikov_controllability_1996}.

The remainder of this paper is organized as follows. In Section~2, we introduce the fundamental concepts of OT and present several key lemmas that will be used later. Sections~3 and~4 are devoted to the main global stability theorems for elliptic and parabolic equations, respectively. In Section~5, we conclude the paper and discuss possible directions for future work. Finally, Section~6 contains the proofs of the lemmas introduced in Section~2.

\section{Preliminaries}
\subsection{Optimal transport and Kantorovich duality}\label{sub: OT}

We recall basic notions from optimal transport required in what follows, including the optimal transport cost, the Kantorovich dual formulation, and a useful reduction of the dual search space. The presentation follows~\cite{villani2003topics}. 
 
\begin{thm}[Optimal transport cost and Kantorovich duality]\label{KV thm}

    Let $X$ and $Y$ be Polish spaces, let $\mu \in P(X)$ and $\nu \in P(Y)$, and let $c: X \times Y \to \mathbb{R}_{+} \cup \{+\infty\}$ be a lower semicontinuous cost function. 

    For $\pi \in P(X \times Y)$ and $(\varphi, \psi) \in L^1(d\mu) \times L^1(d\nu)$, define
    $$
    I[\pi]=\int_{X \times Y} c(x, y) d \pi(x, y), \quad J(\varphi, \psi)=\int_X \varphi d \mu+\int_Y \psi d \nu.
    $$

    Let $\Pi(\mu, \nu)$ denote the set of all Borel probability measures $\pi$ on $X \times Y$ such that for every measurable $A \subset X$ and $B \subset Y$,
    $$
    \pi[A \times Y]=\mu[A], \qquad \pi[X \times B]=\nu[B].
    $$
    The optimal transportation cost between $\mu$ and $\nu$ is then defined by
    $$
    \mathcal{T}_c(\mu, \nu)=\inf_{\pi \in \Pi(\mu, \nu)} I[\pi].
    $$
    Define $\Phi_c$ as the set of all measurable pairs $(\varphi, \psi) \in L^1(d\mu)\times L^1(d\nu)$ satisfying
    $$
    \varphi(x)+\psi(y)\le c(x,y)
    \quad\text{for } \mu\text{-a.e. }x\in X,\ \nu\text{-a.e. }y\in Y.
    $$
    Then the following duality holds:
    \begin{equation}\label{KV problem}
    \inf_{\pi \in \Pi(\mu, \nu)} I[\pi]
    =\sup_{(\varphi,\psi)\in\Phi_c} J(\varphi, \psi).
    \end{equation}
    Moreover, the infimum in the primal problem is attained. 

    If, in addition, there exist nonnegative measurable functions $c_X \in L^1(d\mu)$ and $c_Y \in L^1(d\nu)$ such that
    \begin{equation}\label{maximizer condition}
    c(x, y) \le c_X(x)+c_Y(y)
    \quad \forall (x, y)\in X\times Y,
    \end{equation}
    then the supremum in \eqref{KV problem} is also attained.

\end{thm}

In our setting, $X=Y$ is either $\Omega$ or $Q$, depending on whether the underlying problem is elliptic or parabolic. Throughout, we assume that the cost function $c:X\times X\to\mathbb{R}$ is bounded. This boundedness allows reduction of the dual search space, as noted in~\cite[Remark~1.13]{villani2003topics} and formalized below. We also observe that boundedness of $c$ implies~\eqref{maximizer condition}, ensuring existence of a maximizer in the dual formulation.

\begin{lem}\label{restrict lem}

    Let $\|c\|_\infty=\sup_{x,y}c(x,y)$ and define
    \[
    \widetilde{\Phi}_c=\{(\varphi,\psi)\in\Phi_c:\ 0\le \varphi\le \|c\|_\infty,\ -\|c\|_\infty\le \psi\le 0\}.
    \]
    Then
    \[
    \max\{J(\varphi,\psi):(\varphi,\psi)\in\Phi_c\}
    =\max\{J(\varphi,\psi):(\varphi,\psi)\in\widetilde{\Phi}_c\}.
    \]
\end{lem}

A classical example of a cost function is $c(x,y)=d(x,y)^p$ on a metric space $(X,d)$ with $p\ge1$. In this case, the optimal transport cost induces the  Wasserstein-$p$ distance:
\[
W_p(\mu,\nu)=\Bigl(\inf_{\pi\in\Pi(\mu,\nu)}\int_{X\times X}d(x,y)^p\,d\pi\Bigr)^{1/p},
\qquad \mu,\nu\in\mathcal{P}_p(X),
\]
where $\mathcal{P}_p(X)$ denotes the set of probability measures on $X$ with finite $p$-th moments.

\subsection{CGO solutions and test function spaces}

We now introduce the test function spaces used throughout the analysis. These consist of adjoint solutions and will later be specialized using CGO constructions.

For the elliptic case, define
\[
\mathcal{V}:=\{v\in H^2(\Omega): -\nabla\cdot(\kappa(x)\nabla v)+q(x)v=0 \text{ in }\Omega\}.
\]
For the parabolic case, we use time-dependent functions
\[
\mathcal{V}_t:=\{v\in L^2(0,T;H^1(\Omega)):\ \partial_t v+\Delta v-qv=0\text{ in }Q\}.
\]

A key step in the reconstruction is the construction of test functions $\{v_j\}_{j=1}^n\subset \mathcal{V}$ associated with distinct points $\{s_j\}_{j=1}^n\subset\Omega$, satisfying the interpolation property
\[
v_j(s_\ell)=\delta_{j\ell}.
\]
The following lemma ensures the existence of auxiliary functions with linearly independent point evaluations.

\begin{lem}\label{invertible matrix thm}
      For any collection of \( n \) distinct points \( s_1, \ldots, s_n \in \Omega \), let
\[
\eta_1 := \min_{i \neq j} |s_i - s_j|, \quad \eta_2 := \min_{1 \leq j \leq n} |s_j|.
\]
Then there exist \( n \) test functions \( \tilde{v}_1(x), \ldots, \tilde{v}_n(x) \in \mathcal{V} \) such that the matrix \( A = [a_{l,j}]_{1 \leq l,j \leq n} \), defined by
\begin{equation} \label{matrix A}
    a_{l,j} := \tilde{v}_j(s_l),
\end{equation}
is invertible. Moreover, the following estimates hold:
\begin{equation} \label{bound of A}
    \|A^{-1}\|_2 \leq \frac{\exp(-\tilde{r}\eta_2^2)}{1 - n\left( \exp\left(-\frac{\tilde{r} \eta_1^2}{2}\right) + \frac{C_1 \|\tilde{q}\|_{H^p(\Omega)}}{\sqrt{2} \tilde{r} \eta_2} \right)},
\end{equation}
and
\begin{equation} \label{bound of patial v}
    \max_{1 \leq j \leq n} \left\{ \|\tilde{v}_j\|_{H^1(\Omega)} \right\} \leq C_2 \tilde{r} e^{\tilde{r} R_0^2} \left( 1 + \frac{C_1 \|\tilde{q}\|_{H^p(\Omega)}}{\sqrt{2} \tilde{r} \eta_2} \right),
\end{equation}
where \( C_1 = C_1(\Omega, p) \) and \( C_2 = C_2(\Omega, d) \) are constants independent of the point set \( \{s_j\}_{1 \leq j \leq n} \), $R_0=\max_{x\in \overline{\Omega}}|x|$ and
\[
\tilde{r} := 2n \left( \frac{2}{\eta_1^2} + \frac{1 + C_1 \|\tilde{q}\|_{H^p(\Omega)}}{\sqrt{2} \eta_2} \right), \quad \tilde{q}(x) := \frac{q(x)}{\kappa(x)} + \frac{\Delta \sqrt{\kappa(x)}}{\sqrt{\kappa(x)}}.
\]
\end{lem}

Lemma~\ref{invertible matrix thm} will be proved in ~\ref{appendix:proof of lemma3}, where CGO solutions are used in the construction.

\subsection{Boundary control theory}

To eliminate interior contributions arising in the parabolic setting, we make use of boundary controllability results, whose proof is deferred to ~\ref{appendix:proof of lemma4}. Relying on the observability estimates in~\cite{fursikov_controllability_1996}, we obtain the following lemma.

\begin{lem}\label{lemma:control estimate}
Suppose $v\in \mathcal{V}_t$. Then there exists a boundary control 
\[
\omega\in L^2\bigl(T^*,T;H^{1/2}(\partial\Omega)\bigr)
\]
such that the system
\begin{equation}\label{control system}
\left\{
\begin{aligned}
&\partial_t \psi + \Delta \psi - q\psi = 0 && \text{in } [T^*,T]\times\Omega,\\
&\frac{\partial\psi}{\partial n} = \omega && \text{on } [T^*,T]\times\partial\Omega,\\
&\psi(\cdot,T^*) = v(\cdot,T^*) && \text{in } \Omega,\\
&\psi(\cdot,T)=0 && \text{in } \Omega
\end{aligned}
\right.
\end{equation}
admits a solution $\psi$. Moreover, the control satisfies the estimate
\begin{equation}\label{control estimate}
\|\omega\|_{L^2(\Sigma^+)}
\le C(\Omega,q)\,\|v(\cdot,T)\|_{H^1(\Omega)}
+ \left\|\tfrac{\partial v}{\partial n}\right\|_{L^2(\Sigma^+)}.
\end{equation}
\end{lem}

\section{Case for time-independent equations}

\subsection{Problem formulation}

We first consider the inverse source problem associated with the stationary elliptic equation
\begin{equation}\label{directmodel}
\left\{
\begin{aligned}
-\nabla\!\cdot\!\bigl(\kappa(x)\nabla u\bigr) + q(x)\,u &= \sum_{j=1}^{m} a_j\,\delta(x-s_j) && \text{in }\Omega,\\
\frac{\partial u}{\partial n} &= 0 && \text{on }\partial\Omega,
\end{aligned}
\right.
\end{equation}
where $\Omega\subset\mathbb{R}^d$ is a smooth bounded domain. We assume $\kappa\in H^{p+2}(\Omega)\cap L^\infty(\Omega)$, $q\in H^{p}(\Omega)\cap L^\infty(\Omega)$, and $p>1+\frac{d}{2}$. In addition, $\kappa$ is strictly positive on $\overline\Omega$. The available data consist of boundary measurements of the solution, that is, $u|_{\partial\Omega}$.

We impose the following assumptions on the source term:
\begin{Henum}
\item\label{H1}
The point locations $s_j\in\mathring{\Omega}$ are distinct and the number of points satisfies $m\le M$ for a prescribed integer $M$.
\item\label{H2}
The amplitudes satisfy $a_j>0$ and $\sum_{j=1}^m a_j = 1$.
\end{Henum}

Under~\ref{H1}-\ref{H2}, the measure
\[
\mu = \sum_{j=1}^{m} a_j\delta(x-s_j)
\]
defines a probability measure on $\Omega$. We denote by
\[
\Theta_M := \Bigl\{\mu=\textstyle\sum_{j=1}^{m} a_j \delta(x-s_j): \mu \text{ satisfies \ref{H1}--\ref{H2}}\Bigr\}
\]
the collection of all admissible source measures.

Let $\mu,\nu\in\Theta_M$, and write
\[
\mu=\sum_{s\in N} a_s \delta(x-s),\qquad 
\nu=\sum_{s'\in N'} b_{s'} \delta(x-s'),
\]
where $N,N'$ are finite subsets of $\mathring{\Omega}$. Let $u_1$ and $u_2$ denote the solutions to~\eqref{directmodel} corresponding to $\mu$ and $\nu$, respectively.

Given a bounded cost function $c(x,y)$, our objective is to establish a global stability estimate of the form
\begin{equation}\label{stability}
\mathcal{T}_c(\mu,\nu) \le C\,\|u_1-u_2\|_{L^2(\partial\Omega)},
\end{equation}
where $\mathcal{T}_c$ denotes the optimal-transport cost associated with $c$.

\subsection{Construction of test functions}

For any \(v\in\mathcal{V}\), define the linear functionals
\begin{equation}\label{Definition of Rv}
\mathcal{R}_i(v):=\int_{\partial\Omega}\kappa\,\frac{\partial v}{\partial n}\,u_i\,ds,\qquad i=1,2,
\end{equation}
where $u_i$ solves~\eqref{directmodel} with source $\mu$ (for $i=1$) or $\nu$ (for $i=2$).
Multiplying~\eqref{directmodel} by $v\in\mathcal{V}$ and integrating by parts yields the following identity.

\begin{thm}\label{Rv duality thm}
For all \(v\in\mathcal{V}\),
\[
\mathcal{R}_1(v)=\sum_{s\in N}a_s\,v(s)=\int_\Omega v\,d\mu,\qquad
\mathcal{R}_2(v)=\sum_{s'\in N'}b_{s'}\,v(s')=\int_\Omega v\,d\nu.
\]
\end{thm}

The functionals $\mathcal{R}_i$ have been used in earlier reconstruction approaches (see, e.g.,~\cite{ElBadiaHa-Duong2000, qiu_traceability_2025}) to form linear systems for recovering point-source parameters. Their structure also mirrors the Kantorovich dual formulation and therefore provides a natural bridge to optimal-transport stability estimates.

Indeed, if $(\varphi,\psi)\in\widetilde{\Phi}_c$ and both functions belong to $\mathcal{V}$, then
\begin{equation}\label{Rv and J}
J(\varphi,\psi)=\mathcal{R}_1(\varphi)+\mathcal{R}_2(\psi).
\end{equation}
For general $(\varphi,\psi)\in\widetilde{\Phi}_c$, relation~\eqref{Rv and J} may fail since $\varphi,\psi$ need not lie in $\mathcal{V}$. However, due to the atomic structure of $\mu$ and $\nu$, it suffices to construct functions $v_\varphi,v_\psi\in\mathcal{V}$ satisfying
\begin{equation} \label{vvarphi and vpsi}
v_\varphi(s) = \varphi(s) \quad \forall s \in N, \qquad 
v_\psi(s') = \psi(s') \quad \forall s' \in N'.
\end{equation}
Under this condition, the identity
\[
J(\varphi, \psi) = \mathcal{R}_1(v_\varphi) + \mathcal{R}_2(v_\psi)
\]
holds. The existence of \(v_\varphi\) and \(v_\psi\) is shown in Theorem \ref{two v thm}.

To construct such interpolating functions, we first build a basis $\{v_s\}_{s\in N}\subset\mathcal{V}$ satisfying
\[
v_s(s')=
\begin{cases}
1 & \text{if } s=s',\\
0 & \text{if } s\ne s'.
\end{cases}
\]
Lemma~\ref{invertible matrix thm} guarantees the existence of auxiliary CGO-based functions with linearly independent evaluations, from which the interpolating basis can be constructed. Moreover, the estimates of $\|\tilde{v}\|_{H^1(\Omega)}$ and $\|A^{-1}\|_2$ provide bounds on the norms of these basis functions, which will be used in establishing the main stability result.

\begin{thm}\label{basis thm}
There exist functions $v_1,\dots,v_n\in\mathcal{V}$ such that
\begin{equation}\label{delta condition}
v_i(s_j) = \delta_{ij} \quad \text{for all } 1 \leq i, j \leq n.
\end{equation}
and
\begin{equation}\label{Bound of v}
\max_{1\le i\le n}\|v_i\|_{H^1(\Omega)}
\le C_3\,\sqrt{n}\,\tilde r\,e^{\tilde r(R_0^2-\eta_2^2)},
\end{equation}
where $C_3=C_3(\Omega,d)$ is independent of the point set \( \{s_j\}_{1 \leq j \leq n} \).
\end{thm}

\begin{proof}
Let \( \{\tilde{v}_j(x)\}_{1 \leq j \leq n} \) be the test functions constructed in Lemma~\ref{invertible matrix thm}, and define
\[
\begin{pmatrix}
v_1(x) \\
\vdots \\
v_n(x)
\end{pmatrix}
:= A^{-1}
\begin{pmatrix}
\tilde{v}_1(x) \\
\vdots \\
\tilde{v}_n(x)
\end{pmatrix}.
\]
Then the collection \( \{v_j(x)\}_{1 \leq j \leq n} \subset \mathcal{V} \) satisfies the interpolation condition~\eqref{delta condition}.

To estimate the norms, we use the bound
\begin{equation}\label{linear bound of v}
    \sum_{i=1}^n \|v_i\|_{H^1(\Omega)}^2 \leq \|A^{-1}\|_2^2 \sum_{j=1}^n \|\tilde{v}_j\|_{H^1(\Omega)}^2.
\end{equation}
From the definition of \( \tilde{r} \), we observe that
\[
n\left( \exp\left( -\frac{\tilde{r} \eta_1^2}{2} \right) + \frac{C_1 \|\tilde{q}\|_{H^p(\Omega)}}{\sqrt{2} \tilde{r} \eta_2} \right) \leq \frac{1}{2}.
\]
Substituting the estimates from~\eqref{bound of A} and~\eqref{bound of patial v} into~\eqref{linear bound of v}, and taking the maximum over \( i \), we obtain \eqref{Bound of v}.
\end{proof}

Using the interpolating basis, we may now reformulate the Kantorovich dual problem in terms of the functionals $\mathcal{R}_1$ and $\mathcal{R}_2$.

\begin{thm}\label{two v thm}
For all $(\varphi,\psi)\in\widetilde{\Phi}_c$, there exist test functions $v_{\varphi},v_{\psi}\in\mathcal{V}$ satisfying~\eqref{vvarphi and vpsi}. Moreover,
\[
\mathcal{T}_c(\mu,\nu)=\max_{(v_\varphi,v_\psi)}\bigl\{\mathcal{R}_1(v_\varphi)+\mathcal{R}_2(v_\psi)\bigr\}.
\]
\end{thm}

\begin{remark}
Theorem~\ref{two v thm} shows that the values of the functionals $\mathcal{R}_1$ and $\mathcal{R}_2$ uniquely determine the optimal transport cost $\mathcal{T}_c(\mu,\nu)$. Although the interpolating functions $v_\varphi$ and $v_\psi$ may not be unique, the corresponding values $\mathcal{R}_1(v_\varphi)$ and $\mathcal{R}_2(v_\psi)$ are invariant.
\end{remark}
\subsection{Global OT stability for elliptic point sources}\label{Main for elliptic}

We now turn to the proof of the stability estimate~\eqref{stability}.  
Instead of using a pair of test functions \( (v_\varphi, v_\psi) \in \mathcal{V}\times\mathcal{V} \) as in Theorem~\ref{two v thm}, our strategy is to construct a single function \(v\in\mathcal{V}\) that encodes the discrepancy between \(u_1\) and \(u_2\) while retaining the structure of the Kantorovich dual pairing.  
This reduction is essential, since only a single function appears in the boundary integral representation~\eqref{Definition of Rv}.

To proceed, we decompose the sets of point sources according to how masses are redistributed between \(\mu\) and \(\nu\).  
Define
\[
S_1 := N\setminus N', \qquad 
S_2 := N'\setminus N,
\]
and partition the common support \(N\cap N'\) into
\[
S_3 := \{s\in N\cap N' : a_s>b_s\}, \qquad 
S_4 := (N\cap N')\setminus S_3.
\]
In particular, points in \(S_1\cup S_3\) represent locations where mass must be transported away, while those in \(S_2\cup S_4\) represent locations receiving transported mass.

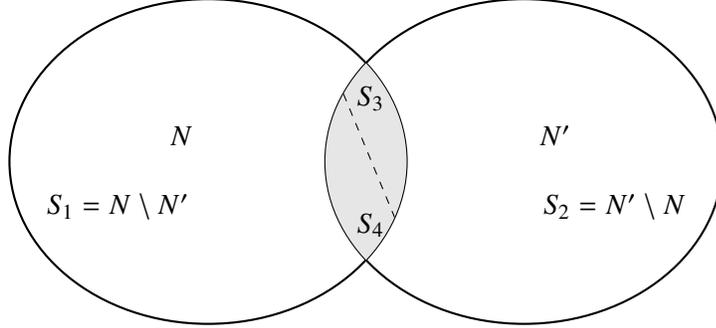
\begin{tikzpicture}[scale=1.2, every node/.style={font=\small}]
\begin{scope}[shift={(1,0)}]
% Draw sets N and N'
\draw[thick] (0,0) ellipse (2.2 and 1.8) node[above left=2pt] {$N$};
\draw[thick] (3.5,0) ellipse (2.2 and 1.8) node[above right=2pt] {$N'$};

% Regions
\node at (-1,-0.5) {$S_1 = N \setminus N'$};
\node at (4.5,-0.5) {$S_2 = N' \setminus N$};

% Overlap area

  \clip (0,0) ellipse (2.2 and 1.8);
  \fill[gray!20] (3.5,0) ellipse (2.2 and 1.8);

% Dashed diagonal to divide S_3 and S_4
\draw[dashed] (1.5,0.75) -- (2.1,-0.7);

% Add S3 and S4 labels
\node at (1.8,0.7) {$S_3$};
\node at (1.8,-0.7) {$S_4$};
\end{scope}
% Legend
\begin{scope}[shift={(-0.5,-2.5)}]
  \draw[gray!20, fill=gray!20] (-0.5,0) rectangle (-1.0,0.3);
  \node[right=5pt] at (-0.6,0.15) {Common points $N \cap N'$};

  \draw[dashed] (3,0.15) -- +(0.5,0);
  \node[right=5pt] at (3.5,0.15) {Division: $S_3$ ($a_s > b_s$), $S_4$ ($a_s \leq b_s$)};
\end{scope}

\end{tikzpicture}

Clearly, the sets \( S_1, S_2, S_3, S_4 \) are pairwise disjoint, and we have the decompositions
\[
N = S_1 \cup S_3 \cup S_4, 
\qquad 
N' = S_2 \cup S_3 \cup S_4.
\]
Using the definitions of \( \mu \) and \( \nu \), the functional \( J(\varphi, \psi) \) can be written as
\begin{equation*}
    J(\varphi, \psi) 
    = \sum_{s \in S_1 \cup S_3 \cup S_4} a_s \varphi(s)
      + \sum_{s \in S_2 \cup S_3 \cup S_4} b_s \psi(s).
\end{equation*}

Let \( S := S_1 \cup S_2 \cup S_3 \cup S_4 \), and define
\begin{equation*}
    \eta_1 := \min_{\substack{s \neq s',\\ s, s' \in S}} |s - s'|,
    \qquad 
    \eta_2 := \min_{s \in S} |s|.
\end{equation*}

\begin{lem}\label{combine test functino lem}
    Assume that \( c(x,x) = 0 \) for all \( x \in \Omega \). 
    Then, for any \( (\varphi, \psi) \in \widetilde\Phi_c \), there exists a test function \( v \in \mathcal{V} \) such that
\[
\begin{cases}
v(s) = \varphi(s), & s \in S_1 \cup S_3, \\[0.2em]
v(s) = -\psi(s),   & s \in S_2 \cup S_4.
\end{cases}
\]
Moreover, the following inequality holds:
\begin{equation} \label{J and v}
    J(\varphi, \psi) \leq \mathcal{R}_1(v) - \mathcal{R}_2(v).
\end{equation}
\end{lem}
\begin{proof}

    The existence of \( v \) follows from the construction of the basis functions in Theorem~\ref{basis thm}. 
   More precisely, we first construct a collection \( \{v_s\}_{s \in S} \subset \mathcal{V} \) such that
    \[
    v_s(s') = 
    \begin{cases}
    1, & \text{if } s = s', \\[0.2em]
    0, & \text{if } s \neq s',
    \end{cases}
    \qquad s,s'\in S.
    \]
    Then the desired test function is given by
    \[
    v = \sum_{s \in S_1 \cup S_3} \varphi(s)\, v_s
        - \sum_{s \in S_2 \cup S_4} \psi(s)\, v_s.
    \]

    Since \( (\varphi, \psi) \in \widetilde\Phi_c \), we have 
    \(|\varphi(s)| \le \|c\|_\infty\) and \(|\psi(s)| \le \|c\|_\infty\) for all \(s\in S\). Hence
    \begin{equation}\label{estimate of v on boundary}
    \begin{aligned}
    \left\| \frac{\partial v}{\partial n} \right\|_{L^2(\partial \Omega)} 
    &\leq \left( \sum_{s \in S_1 \cup S_3} |\varphi(s)| 
            + \sum_{s \in S_2 \cup S_4} |\psi(s)| \right) \,\,
        \max_{s \in S} \left\| \frac{\partial v_s}{\partial n} \right\|_{L^2(\partial \Omega)} \\
    &\leq \|c\|_\infty |S| \,\, \max_{s \in S} \left\| \frac{\partial v_s}{\partial n} \right\|_{L^2(\partial \Omega)} \\
    &\leq C(\Omega)\,\|c\|_\infty |S| \,\, \max_{s \in S} \|v_s\|_{H^1(\Omega)} \\
    &\leq C_3(\Omega,d)\,\|c\|_\infty |S|\, \tilde{r}\, e^{\tilde{r}(R_0^2 - \eta_2^2)},
    \end{aligned}
    \end{equation}
    where now
    \[
    \tilde{r} := 2|S| \left( \frac{2}{\eta_1^2} 
        + \frac{1 + C_1 \|\tilde{q}\|_{H^p(\Omega)}}{\sqrt{2} \eta_2} \right),
    \]
    and \(C_1\) and \(C_3\) are as in Lemma~\ref{invertible matrix thm} and Theorem~\ref{basis thm}.

    Next, for any \( s \in S_3 \cup S_4 \), the admissibility condition
    \[
    \varphi(s) + \psi(s) \leq c(s, s) = 0
    \]
    implies \(\varphi(s) \le -\psi(s)\) on \(S_3\) and \(\psi(s) \le -\varphi(s)\) on \(S_4\). 
    Define
    \[
    w(s) := 
    \begin{cases}
    \varphi(s), & \text{if } s \in S_3,\\[0.2em]
    \psi(s),    & \text{if } s \in S_4.
    \end{cases}
    \]
    By the definition of \( v \), we then have
    \[
    a_s v(s) - b_s v(s) = |a_s - b_s|\, w(s), 
    \qquad \forall s \in S_3 \cup S_4.
    \]

    We now rewrite \(J(\varphi,\psi)\) using the above decomposition:
    \begin{align*}
    J(\varphi, \psi) 
    &= \sum_{s \in S_1 \cup S_3 \cup S_4} a_s \varphi(s)
    + \sum_{s \in S_2 \cup S_3 \cup S_4} b_s \psi(s) \\
    &= \sum_{s \in S_3 \cup S_4} 
        \Bigl[\min\{a_s, b_s\}\bigl(\varphi(s) + \psi(s)\bigr)
                + |a_s - b_s|\, w(s) \Bigr] \\
    &\quad + \sum_{s \in S_1} a_s \varphi(s)
            + \sum_{s \in S_2} b_s \psi(s).
    \end{align*}
    By the inequality \(\varphi(s)+\psi(s)\le 0\) on \(S_3\cup S_4\), we obtain
    \begin{align*}
    J(\varphi, \psi) 
    &\leq \sum_{s \in S_3 \cup S_4} |a_s - b_s|\, w(s)
        + \sum_{s \in S_1} a_s \varphi(s)
        + \sum_{s \in S_2} b_s \psi(s) \\
    &= \sum_{s \in N} a_s v(s) - \sum_{s' \in N'} b_{s'} v(s') \\
    &= \mathcal{R}_1(v) - \mathcal{R}_2(v),
    \end{align*}
    which proves~\eqref{J and v}.
\end{proof}

\begin{remark}
    The inequality~\eqref{J and v} may be strict even if \( (\varphi^*, \psi^*) \) is an optimal pair attaining the supremum of \( J(\varphi, \psi) \). 
    However, when \( c \) is a distance function on \( \Omega \), one can choose an optimal pair \( (\varphi^*, \psi^*) \in \widetilde{\Phi}_c \) such that
    \[
    \varphi^*(s) + \psi^*(s) = 0 
    \quad \text{for all } s \in S_3 \cup S_4,
    \]
    so that equality holds in~\eqref{J and v}. 
    This follows, for instance, from Theorem~1.14 in~\cite{villani2003topics}.
\end{remark}

We are now in a position to state the main stability theorem.
\begin{thm}
Let $\mu$ and $\nu$ be two probability measures in $\Theta_M$. 
Then there exists a test function $v^*\in\mathcal{V}$ such that
\begin{equation}\label{main result}
    \mathcal{T}_c(\mu, \nu) 
    \leq \bigl|\mathcal{R}_1(v^*)-\mathcal{R}_2(v^*)\bigr|
    \leq C_3(\Omega,d)\,\|\kappa\|_{C^0(\bar{\Omega})}\,\|c\|_{\infty}\,
        M\,\tilde{r}\, e^{\tilde{r}(R_0^2 - \eta_2^2)}\,
        \|u_1 - u_2\|_{L^2(\partial \Omega)},
\end{equation}
where
\[
\tilde{r} = 4M \left( \frac{2}{\eta_1^2} 
      + \frac{1 + C_1 \|\tilde{q}\|_{H^p(\Omega)}}{\sqrt{2} \eta_2} \right),
\]
and \(C_1\) is as in Lemma~\ref{invertible matrix thm}.
\end{thm}

\begin{proof}
    By Kantorovich duality and Lemma~\ref{restrict lem}, we have
    \[
        \mathcal{T}_c(\mu, \nu) 
        = \max_{(\varphi, \psi) \in \widetilde{\Phi}_c} J(\varphi, \psi).
    \]
    Let \( (\varphi^*, \psi^*) \in \widetilde{\Phi}_c \) be a maximizer, and apply Lemma~\ref{combine test functino lem} to obtain a test function \( v^* \in \mathcal{V} \) such that
    \begin{equation}\label{J and v recall}
        J(\varphi^*, \psi^*) \leq \mathcal{R}_1(v^*) - \mathcal{R}_2(v^*).
    \end{equation}    
    On the other hand, from the definition of \( \mathcal{R}_i(v) \) in~\eqref{Definition of Rv}, we have, for every \(v\in\mathcal{V}\),
    \begin{equation}\label{bound of Rv}
        |\mathcal{R}_1(v) - \mathcal{R}_2(v)| 
        \leq \|\kappa\|_{C^0(\bar{\Omega})}
            \left\| \frac{\partial v}{\partial n} \right\|_{L^2(\partial \Omega)} 
            \|u_1 - u_2\|_{L^2(\partial \Omega)}.
    \end{equation}
    Moreover, for \(v=v^*\), the uniform bound on \( \left\| \frac{\partial v^*}{\partial n} \right\|_{L^2(\partial \Omega)} \) is provided by~\eqref{estimate of v on boundary}. 
    Combining~\eqref{J and v recall},~\eqref{bound of Rv}, and~\eqref{estimate of v on boundary}, we obtain
    \[
    J(\varphi^*, \psi^*) 
    \leq C_3(\Omega,d)\,\|\kappa\|_{C^0(\bar{\Omega})}\,\|c\|_{\infty}\,
            |S|\,\tilde{r}\, e^{\tilde{r}(R_0^2 - \eta_2^2)}\,
            \|u_1 - u_2\|_{L^2(\partial \Omega)}.
    \]
    Since \( \mu, \nu \in \Theta_M \), we have \( |S| \leq 2M \), and the choice of \(\tilde r\) with \(4M\) in its definition yields~\eqref{main result}. 
    Finally, using \(\mathcal{T}_c(\mu,\nu)=J(\varphi^*,\psi^*)\) completes the proof.
\end{proof}

\section{Case for parabolic equations}

We now consider the time-dependent setting and study the inverse point-source problem for a parabolic partial differential equation. Our goal is to derive a global stability result, in the sense of optimal transport, analogous to the elliptic estimate~\eqref{stability}.

\subsection{Problem setting}

We begin with the following forward model:
\begin{align}\label{directmodel 1}
\left\{
\begin{array}{ccll}
  u_t - \Delta u + q(x)\,u &=& \displaystyle\sum_{j=1}^{m} g_j(t)\,\delta(x-s_j)
  & \text{in } Q := \Omega\times(0,T), \\[0.4em]
  \dfrac{\partial u}{\partial n} &=& 0 
  & \text{on } \Sigma := \partial \Omega\times(0,T),\\[0.4em]
  u(x,0) &=& 0 
  & \text{in } \Omega.
\end{array}
\right.
\end{align}

\begin{remark}
A more general model with a spatially varying diffusion coefficient,
\[
    u_t - \nabla\cdot(\kappa(x)\nabla u) + q(x)\,u 
    = \sum_{j=1}^{m} g_j(t)\,\delta(x-s_j),
\]
may also be considered. Since the arguments below change only at the level of notation and not conceptually, we restrict to the case $\kappa(x)\equiv1$ for clarity of exposition.
\end{remark}

We impose the following assumptions on the space--time source.

\begin{Henum}[start=3]
    \item\label{H3} 
    The point source locations \(s_j\in\mathring{\Omega}\) are distinct, and \(m\le M\), where \(M\) is a fixed integer.

    \item\label{H4} 
    The source intensities satisfy \(g_j(t)\ge0\) and \(g_j\in L^2(0,T)\) for \(j=1,\dots,m\). 
    There exists \(T^*\in(0,T)\) such that
    \[
    g_j(t)=0 \quad \text{for } t\in[T^*,T],
    \]
    and the total intensity is normalized:
    \[
    \sum_{j=1}^m \int_{0}^{T^*} g_j(t)\,dt = 1.
    \]

    \item\label{H5}
    Each intensity function \(g_j(t)\) belongs to the finite-dimensional subspace
    \begin{equation}\label{def:GK}
        G_K := \Bigl\{
        g\in L^2(0,T^*) :
        \int_{0}^{T^*} g(t)\,e^{\frac{2k\pi \mathrm{i}}{T^*}t}\,dt = 0
        \ \text{for all } k\in\mathbb{Z},\,|k|>K
        \Bigr\}.
    \end{equation}
    Equivalently, the time dependence of each source is band-limited to Fourier frequencies \(|k|\le K\).
\end{Henum}

Under assumptions~\ref{H3}--\ref{H5}, the source term defines a probability measure on the space--time domain \(Q\). 
We denote the admissible class by
\[
\Theta_{M,K} := 
\Bigl\{\,
\mu = \sum_{j=1}^m g_j(t)\,\delta(x-s_j)
\ \text{such that } \mu \text{ satisfies assumptions \ref{H3}--\ref{H5}}
\Bigr\}.
\]

Let \(\mu,\nu\in\Theta_{M,K}\). 
We write the corresponding decompositions
\[
\mu = \sum_{s\in N} a_s(t)\,\delta(x-s),
\qquad
\nu = \sum_{s'\in N'} b_{s'}(t)\,\delta(x-s'),
\]
and denote by \(u_1,u_2\) the solutions of~\eqref{directmodel 1} associated with \(\mu\) and \(\nu\), respectively.

In the parabolic setting, the underlying optimal transport space is \(Q\). 
Let \(c:Q\times Q\to\mathbb{R}\) be a bounded cost function satisfying the assumptions stated in Section~\ref{sub: OT}. 
For two admissible sources \(\mu,\nu\in\Theta_{M,K}\), the corresponding optimal transport cost will be denoted by \(\mathcal{T}_c(\mu,\nu)\).

Our objective is to establish a global stability estimate relating the boundary discrepancy of the parabolic solutions to the transport geometry:
\begin{equation}\label{para stability}
    \mathcal{T}_c(\mu, \nu) 
    \le C\, \|u_1 - u_2\|_{L^2(\Sigma)},
\end{equation}
where \(C>0\) depends only on \(\Omega\), \(M\), \(K\), and the coefficients of the equation.

\subsection{Construction of time-dependent test functions}

We introduce the following notation associated with the intermediate time \(T^*\):
\begin{align*}
    &Q^{-} := \Omega \times (0, T^*), \quad \Sigma^{-} := \partial\Omega \times (0, T^*), \\
    &Q^{+} := \Omega \times (T^*, T), \quad \Sigma^{+} := \partial\Omega \times (T^*, T).
\end{align*}

For any \(v \in \mathcal{V}_t\), and for the two solutions \(u_1,u_2\) corresponding to \(\mu\) and \(\nu\), respectively, we define
\begin{equation}\label{parafunctionalR}
    \mathcal{R}_i(v) 
    := \int_\Omega u_i(x, T^*)\, v(x, T^*) \, dx
       + \int_{\Sigma^-} u_i\,\frac{\partial v}{\partial n} \, d\sigma,
    \qquad i=1,2.
\end{equation}

Multiplying~\eqref{directmodel 1} by a test function \(v \in \mathcal{V}_t\) and integrating by parts over \(Q^-\) yields the following representation of \(\mathcal{R}_i\).

\begin{thm}
    For all \(v\in \mathcal{V}_t\), we have
    \begin{align*}
        \mathcal{R}_1(v)
            &= \sum_{s\in N} \int_{0}^{T} a_s(t)\,v(s,t)\,dt
             = \int_{Q} v\,d\mu,\\
        \mathcal{R}_2(v)
            &= \sum_{s^{\prime}\in N^{\prime}}\int_{0}^{T} b_{s^{\prime}}(t)\,v(s^{\prime},t)\,dt
             = \int_{Q} v\,d\nu.
    \end{align*}
\end{thm}

The identities above show that, as in the elliptic case, the functionals \(\mathcal{R}_1\) and \(\mathcal{R}_2\) depend only on the values of \(v\) at the spatial source locations, weighted by the corresponding time-dependent intensities.
We now construct time-dependent test functions by combining the spatial basis from Theorem~\ref{basis thm} with a truncated temporal Fourier expansion.

\begin{thm}\label{para basis thm}
Let \(h_1,\ldots,h_n \in G_K\) and let \(s_1,\ldots,s_n \in \Omega\) be distinct points. Then there exists a test function \(v(x,t)\in\mathcal{V}_t\) such that
\[
v(s_j, t) = h_j(t), \quad \text{for all } 1 \leq j \leq n \text{ and } t \in (0, T^*).
\]
Moreover, there exists a constant
\[
C_4 = C_4(\eta_1,\eta_2,n,\Omega,q,K) > 0,
\]
where 
\[
\eta_1 := \min_{i \neq j} |s_i - s_j|, 
\qquad 
\eta_2 := \min_{1 \le j\le n} |s_j|,
\]
such that
\begin{equation}\label{para Bound of v}
    \left\| \frac{\partial v}{\partial n} \right\|_{L^2(\Sigma^-)} 
    \leq C_4 \,\max_{1 \leq j \leq n} \| h_j \|_{L^2(0, T^*)}.
\end{equation}
\end{thm}
\begin{proof}
    Let \(e_k(t) := e^{\frac{2\pi i k t}{T^*}}\) for \(k\in\mathbb{Z}\) denote the Fourier modes on \((0,T^*)\).
    Since \(h_j \in G_K\), each \(h_j\) admits a truncated Fourier expansion
    \[
        h_j(t) = \sum_{|k|\le K} c_{k,j}\, e_k(t).
    \]

    For each Fourier mode \(k\) and each index \(j\), Theorem~\ref{basis thm} can be applied to the modified potential
    \[
        q_k(x) := q(x) - \frac{2\pi i k}{T^*}
    \]
    to construct spatial test functions \(v_{k,j}(x)\in H^1(\Omega)\) satisfying
    \begin{equation*}
        v_{k,j}(s_\ell) = \delta_{\ell,j},
        \qquad
        (\Delta - q_k)v_{k,j}(x) = 0.
    \end{equation*}
    Moreover, from~\eqref{Bound of v} we have
    \[
        \max_{1\le j\le n} \|v_{k,j}\|_{H^1(\Omega)}
        \le C_3\,\sqrt{n}\,\tilde r_k\, e^{\tilde r_k(R_0^2 - \eta_2^2)},
    \]
    where
    \[
    \tilde r_k 
        := 2n \left( \frac{2}{\eta_1^2} 
            + \frac{1 + C_1 \|q_k\|_{H^p(\Omega)}}{\sqrt{2} \eta_2} \right),
    \]
    and \(C_1\) and \(C_3\) are the constants from Lemma~\ref{invertible matrix thm} and Theorem~\ref{basis thm}, respectively.

    For \(|k|\le K\), we further observe that
    \[
        \|q - \tfrac{2\pi i k}{T^*}\|_{H^p(\Omega)}
        \le \|q\|_{H^p(\Omega)} + \frac{2\pi K}{T^*}\sqrt{|\Omega|},
    \]
    so that
    \[
        \tilde r_k \le r_K 
        := 2n\left(
            \frac{2}{\eta_1^2}
            + \frac{1 + C_1\bigl(\|q\|_{H^p(\Omega)} + \frac{2\pi K}{T^*}\sqrt{|\Omega|}\bigr)}{\sqrt{2}\,\eta_2}
        \right).
    \]
    Consequently,
    \begin{equation}\label{uniform vk bound}
    \max_{\substack{1 \leq j \leq n \\ |k| \leq K}} 
        \|v_{k,j}\|_{H^1(\Omega)} 
        \leq C_3 \sqrt{n}\, r_K\, e^{r_K(R_0^2 - \eta_2^2)}.
    \end{equation}
    
    We now define the time-dependent test function
    \begin{equation}\label{sum of v}
        v(x,t) := \sum_{j=1}^n \sum_{|k| \leq K} c_{k,j}\, e_k(t)\, v_{k,j}(x).
    \end{equation}
    By construction and the choice of \(q_k\), we have
    \[
    \partial_t v + \Delta v - q(x)v = 0 \quad \text{in } Q^{-},
    \]
    so \(v\in\mathcal{V}_t\). Moreover, at the interpolation points,
    \[
    v(s_j,t) 
        = \sum_{|k|\le K} c_{k,j} e_k(t)
        = h_j(t), \qquad 1\le j\le n,\ t\in(0,T^*).
    \]

    To estimate the spatial Sobolev norm of \(v\), we use Cauchy-Schwarz and the orthogonality of the Fourier modes:
    \begin{align*}
    \|v(\cdot,t)\|_{H^1(\Omega)} 
    &\le \left( \sum_{j=1}^{n} \sum_{|k| \leq K} |c_{k,j}|^2 \right)^{1/2}
        \left( \sum_{j=1}^{n} \sum_{|k| \leq K} \|v_{k,j}\|^2_{H^1(\Omega)} \right)^{1/2} \\
    &= \left( \frac{1}{T^*} \sum_{j=1}^n \|h_j\|^2_{L^2(0,T^*)} \right)^{1/2}
        \left( \sum_{j=1}^n \sum_{|k| \leq K} \|v_{k,j}\|^2_{H^1(\Omega)} \right)^{1/2} \\
    &\leq \sqrt{\frac{n^2 (2K+1)}{T^*}}\,
        \max_{1\le j\le n}\|h_j\|_{L^2(0,T^*)}\,
        \max_{\substack{1 \leq j \leq n \\ |k| \leq K}} \|v_{k,j}\|_{H^1(\Omega)}.
    \end{align*}
    Combining this with~\eqref{uniform vk bound} and using standard trace estimates on \(\partial\Omega\) yields
    \[
    \left\| \frac{\partial v}{\partial n} \right\|_{L^2(\Sigma^-)}
        \le C_4(\eta_1,\eta_2,n,\Omega,q,K)\,
            \max_{1\le j\le n}\|h_j\|_{L^2(0,T^*)},
    \]
    which is precisely~\eqref{para Bound of v}.
\end{proof}

\begin{remark}
    The band-limited assumption is essential in the above construction.
    When \(h_j\) are arbitrary functions in \(L^2(0,T^*)\), the above construction does not apply directly. In that case, the series~\eqref{sum of v} need not converge in \(L^2(0,T^*;H^1(\Omega))\): the constant \(C_4\) in~\eqref{para Bound of v} grows exponentially as \(K\to\infty\), and the resulting bound becomes ineffective.  
    If one were able to construct \(v\) for general \(h\in L^2(0,T^*)\) with uniform control on the boundary norm, then the band-limiting assumption \(g_j\in G_K\) could be removed.
\end{remark}

Now we relate the optimal transport cost \(\mathcal{T}_c(\mu,\nu)\) to the boundary functionals \(\mathcal{R}_1\) and \(\mathcal{R}_2\).

\begin{thm}\label{para two v thm}
Let \( (\varphi, \psi) \in \widetilde{\Phi}_c \). Then, using the basis functions constructed in Theorem~\ref{para basis thm}, one can construct test functions \( v_\varphi, v_\psi \in \mathcal{V}_t \) such that
\begin{equation}\label{para vvarphi and vpsi}
    J(\varphi, \psi) = \mathcal{R}_1(v_\varphi) + \mathcal{R}_2(v_\psi).
\end{equation}
Moreover, the optimal transport cost satisfies
\begin{equation}\label{para v and cost}
    \mathcal{T}_c(\mu, \nu) = \max_{(v_\varphi, v_\psi)} 
        \left\{ \mathcal{R}_1(v_\varphi) + \mathcal{R}_2(v_\psi) \right\},
\end{equation}
where the maximum is taken over all pairs \( (v_\varphi, v_\psi) \in \mathcal{V}_t \times \mathcal{V}_t \) corresponding to admissible dual potentials \( (\varphi, \psi) \in \widetilde{\Phi}_c \) through~\eqref{para vvarphi and vpsi}.
\end{thm}

\begin{proof}
    By the definitions of \(\mu\) and \(\nu\), we may write
    \begin{equation}\label{definition of J  para}
        J(\varphi, \psi)
        = \sum_{s \in N} \int_0^{T^*} \varphi(s,t)\, a_s(t)\,dt
        + \sum_{s' \in N'} \int_0^{T^*} \psi(s',t)\, b_{s'}(t)\,dt.
    \end{equation}

    Let \(\mathcal{P}: L^2(0, T^*) \to G_K\) denote the orthogonal projection onto \(G_K\).  
    Then, for any \(g\in L^2(0,T^*)\) and any \(g'\in G_K\),
    \[
    \langle g - \mathcal{P}g,\, g' \rangle_{L^2(0,T^*)} = 0,
    \qquad
    \|\mathcal{P}g\|_{L^2(0,T^*)} \le \|g\|_{L^2(0,T^*)}.
    \]
    Since \(a_s, b_{s'}\in G_K\), we can replace \(\varphi(s,\cdot)\) and \(\psi(s',\cdot)\) by their projections without changing the pairing:
    \begin{equation}\label{projection identity}
        \sum_{s \in N} \int_0^{T^*} \varphi(s,t)\, a_s(t)\,dt
        = \sum_{s\in N} \int_0^{T^*} \mathcal{P}\bigl(\varphi(s,\cdot)\bigr)(t)\, a_s(t)\,dt,
    \end{equation}
    and similarly for the terms involving \(\psi\) and \(\nu\).

    Applying Theorem~\ref{para basis thm}, we can construct the corresponding test functions \( v_\varphi \) and \( v_\psi \) such that
    \begin{equation*}
        v_\varphi(s, t) = \mathcal{P}(\varphi(s, t)), \quad 
        v_\psi(s', t) = \mathcal{P}(\psi(s', t)) 
        \quad \text{for all } s \in N,\; s' \in N'.
    \end{equation*}

    Substituting these identities into~\eqref{definition of J  para} and using~\eqref{projection identity}, together with the representation of \(\mathcal{R}_i\) in terms of the sources, we obtain \eqref{para vvarphi and vpsi}. The characterization~\eqref{para v and cost} then follows directly from Kantorovich duality combined with Lemma~\ref{restrict lem}.
\end{proof}

\begin{remark}
    As in Theorem~\ref{two v thm} for the elliptic case, Theorem~\ref{para two v thm} shows that the values of the functionals \( \mathcal{R}_1 \) and \( \mathcal{R}_2 \) uniquely determine the optimal transport cost \( \mathcal{T}_c(\mu, \nu) \) within the admissible class \(\Theta_{M,K}\).  
    Although the test functions \(v_\varphi\) and \(v_\psi\) realizing~\eqref{para vvarphi and vpsi} are not unique, the corresponding values \(\mathcal{R}_1(v_\varphi)\) and \(\mathcal{R}_2(v_\psi)\) are invariant.
\end{remark}

\subsection{Global OT stability for parabolic point sources}

We now establish the global stability estimate \eqref{para stability} for parabolic point sources.
Let $\mu,\nu\in\Theta_{M,K}$ with spatial supports
$N$ and $N'$ and write
\[
S_1 := N \setminus N', \quad 
S_2 := N' \setminus N, \quad 
S_3 := N \cap N',
\]
and
\[
S := S_1 \cup S_2 \cup S_3.
\]
We also introduce the separation parameters
\[
\eta_1 := \min_{\substack{s \neq s' \\ s, s' \in S}} |s - s'|, \qquad 
\eta_2 := \min_{s \in S} |s|.
\]

\begin{lem}\label{para combine test functino lem}
Assume that the cost function satisfies \(c((x,t),(x,t))=0\) for all \((x,t)\in Q^{-}\).
For any \( (\varphi,\psi)\in\widetilde{\Phi}_c \), define
\[
\tilde{v}(s,t):=
\begin{cases}
\varphi(s,t), & s\in S_1 \ \text{or}\ \bigl(s\in S_3,\ a_s(t)\ge b_s(t)\bigr), \\[0.3em]
-\psi(s,t), & s\in S_2 \ \text{or}\ \bigl(s\in S_3,\ a_s(t)< b_s(t)\bigr),
\end{cases}
\qquad s\in S,\ t\in(0,T^*).
\]
Then there exists a test function \(v\in\mathcal{V}_t\) such that
\[
v(s,\cdot)=\mathcal{P}\bigl(\tilde{v}(s,\cdot)\bigr)\quad\text{for all }s\in S,
\]
where \(\mathcal{P}:L^2(0,T^*)\to G_K\) denotes the orthogonal projection onto \(G_K\).
Moreover,
\begin{equation}\label{para J and v}
    J(\varphi,\psi)\le \mathcal{R}_1(v)-\mathcal{R}_2(v).
\end{equation}
\end{lem}

\begin{proof}

    Since \( (\varphi, \psi) \in \widetilde{\Phi}_c \), both \(\varphi\) and \(\psi\) are bounded on \(Q\). In particular,
    \begin{equation*}
        \|\varphi\|_{L^{\infty}(Q)} \leq \|c\|_{\infty}, 
        \qquad 
        \|\psi\|_{L^{\infty}(Q)} \leq \|c\|_{\infty}.
    \end{equation*}
    Hence, by the definition of \(\tilde{v}\),
    \begin{equation*}
        \|\tilde{v}(s, \cdot)\|_{L^2(0, T^*)} 
        \leq \sqrt{T^*}\,\|c\|_{\infty}
        \quad \text{for all } s \in S.
    \end{equation*}
    Since \(\mathcal{P}\) is the orthogonal projection onto \(G_K\), we also have
    \begin{equation*}
        \|\mathcal{P}(\tilde{v}(s, \cdot))\|_{L^2(0, T^*)} 
        \leq \|\tilde{v}(s, \cdot)\|_{L^2(0, T^*)} 
        \leq \sqrt{T^*}\,\|c\|_{\infty},
        \quad s \in S.
    \end{equation*}

    Applying Theorem~\ref{para basis thm}, we can construct a test function \( v \in \mathcal{V}_t \) such that
    \begin{equation}\label{eq:v-interpolation-S}
        v(s,t) = h_s(t) = \mathcal{P}\big(\tilde{v}(s,\cdot)\big)(t), 
        \quad \text{for all } s \in S,\ t\in(0,T^*).
    \end{equation}

    Moreover, from~\eqref{para Bound of v} and the bound on \(\mathcal{P}(\tilde{v})\), we obtain
    \begin{equation}\label{bound of v H1 para}
        \|v(\cdot,t)\|_{H^{1}(\Omega)}
        \leq C_4 \,\sqrt{T^*}\,\|c\|_{\infty},
        \quad \text{for all } t \in (0,T^*).
    \end{equation}

    Next, for any \(s \in S_3\), the admissibility condition \((\varphi,\psi)\in\widetilde{\Phi}_c\) implies
    \[
    \varphi(s, t) + \psi(s, t) \leq c\big((s, t), (s, t)\big) = 0.
    \]
    Define
    \[
    w(s, t) :=
    \begin{cases}
    \varphi(s, t), & \text{if } a_s(t) \geq b_s(t), \\[0.3em]
    \psi(s, t), & \text{if } a_s(t) < b_s(t),
    \end{cases}
    \qquad s \in S_3.
    \]
    By the definition of \(\tilde{v}\), we then have, for all \(s\in S_3\),
    \[
    a_s(t) \tilde{v}(s,t) - b_s(t) \tilde{v}(s,t) 
    = |a_s(t) - b_s(t)|\, w(s, t),
    \quad t\in(0,T^*).
    \]

    We now estimate \(J(\varphi,\psi)\). Using the definition of \(J\) and splitting the sums over \(S_1,S_2,S_3\), we obtain
    \begin{align*}
    J(\varphi, \psi) 
    &= \sum_{s \in S_1 \cup S_3} \int_0^{T^*} a_s(t) \varphi(s, t) \, dt 
    + \sum_{s \in S_2 \cup S_3} \int_0^{T^*} b_s(t) \psi(s, t) \, dt \\
    &= \sum_{s \in S_3} \int_0^{T^*} 
        \Big(
        \min\{a_s(t), b_s(t)\}\big(\varphi(s, t) + \psi(s, t)\big)
        + |a_s(t) - b_s(t)|\, w(s, t)
        \Big)\, dt \\
    &\quad + \sum_{s \in S_1} \int_0^{T^*} a_s(t) \varphi(s, t) \, dt
        + \sum_{s \in S_2} \int_0^{T^*} b_s(t) \psi(s, t) \, dt.
    \end{align*}
    Since \(\varphi(s, t) + \psi(s, t) \leq 0\) for \(s\in S_3\), the terms involving \(\min\{a_s,b_s\}\) are nonpositive. Hence
    \begin{align*}
    J(\varphi, \psi) 
    &\leq \sum_{s \in S_3} \int_0^{T^*} |a_s(t) - b_s(t)|\, w(s, t)\,dt
    + \sum_{s \in S_1} \int_0^{T^*} a_s(t) \varphi(s, t)\,dt \\
    &\quad + \sum_{s \in S_2} \int_0^{T^*} b_s(t) \psi(s, t)\,dt \\
    &= \int_0^{T^*} \left( 
        \sum_{s \in S_3} |a_s(t) - b_s(t)|\, w(s, t) 
        + \sum_{s \in S_1} a_s(t) \varphi(s, t) 
        + \sum_{s \in S_2} b_s(t) \psi(s, t)
        \right) dt.
    \end{align*}
    By the definition of \(\tilde{v}\), the right-hand side can be rewritten as
    \begin{align*}
    J(\varphi, \psi) 
    &= \int_0^{T^*} \left( 
        \sum_{s \in N} a_s(t) \tilde{v}(s, t) 
        - \sum_{s' \in N'} b_{s'}(t) \tilde{v}(s', t) 
        \right) dt.
    \end{align*}

    Since \(a_s,b_{s'}\in G_K\) and \(\mathcal{P}\) is the orthogonal projection onto \(G_K\), we have, as in~\eqref{projection identity},
    \[
    \int_0^{T^*} a_s(t) \tilde{v}(s,t)\,dt
    = \int_0^{T^*} a_s(t)\,\mathcal{P}\big(\tilde{v}(s,\cdot)\big)(t)\,dt,
    \quad s\in N,
    \]
    and similarly for the terms involving \(b_{s'}\). Therefore,
    \begin{align*}
    J(\varphi,\psi)
    &= \int_0^{T^*} \left( 
        \sum_{s \in N} a_s(t) \mathcal{P}\big(\tilde{v}(s,\cdot)\big)(t) 
        - \sum_{s' \in N'} b_{s'}(t) \mathcal{P}\big(\tilde{v}(s',\cdot)\big)(t) 
        \right) dt \\
    &= \int_0^{T^*} \left( 
        \sum_{s \in N} a_s(t) v(s, t) 
        - \sum_{s' \in N'} b_{s'}(t) v(s', t) 
        \right) dt \\
    &= \mathcal{R}_1(v) - \mathcal{R}_2(v),
    \end{align*}
    where in the last step we used the representation of \(\mathcal{R}_i\) in terms of the sources.  
    This proves~\eqref{para J and v}.

\end{proof}
 
We are now ready to state and prove the main stability estimate.
\begin{thm}
Let \(\mu,\nu \in \Theta_{M,K}\) be two probability measures. Then there exists a test function \(v^*\in\mathcal{V}_t\) such that
\begin{equation}\label{final para result}
    \mathcal{T}_c(\mu, \nu)
    \leq \big|\mathcal{R}_1(v^*)-\mathcal{R}_2(v^*)\big|
    \leq C_5(\eta_1,\eta_2,K,M,\Omega,T^*,T,c)\,
        \|u_1-u_2\|_{L^2(\Sigma)}.
\end{equation}
\end{thm}

The same strategy also yields OT-type stability for inverse point-mass 
initial data problems for parabolic equations; see \ref{appendix:initial-data}.

\begin{proof}
    We first show that for any \(v\in\mathcal{V}_t\),
    \begin{equation}\label{bound of Rv in para}
        |\mathcal{R}_1(v) - \mathcal{R}_2(v)| \leq C(\Omega, q)\left( \|v(T^*)\|_{H^1(\Omega)} + \|v\|_{L^2(0,T; H^1(\Omega))} \right)\|u_1 - u_2\|_{L^2(\Sigma)}.
    \end{equation}
    By definition,
    \begin{equation*}
        \mathcal{R}_1(v) - \mathcal{R}_2(v) 
        = \int_\Omega (u_1(x, T^*) - u_2(x, T^*))\,v(x, T^*)\,dx 
        + \int_{\Sigma^-} (u_1 - u_2)\,\frac{\partial v}{\partial n}\,d\sigma.
    \end{equation*}
    By Lemma~\ref{lemma:control estimate}, for the given trace \(v(\cdot,T^*)\) there exists a pair \((\psi,\omega_v)\) solving the adjoint system~\eqref{control system} on \([T^*,T]\times\Omega\), with \(\psi(\cdot,T^*) = v(\cdot,T^*)\) and such that
    \[
    \|\omega_v\|_{L^2(\Sigma^+)}
    \le C(\Omega,q)\Big(\|v(T^*)\|_{H^1(\Omega)} + \|v\|_{L^2(0,T;H^1(\Omega))}\Big).
    \]
    Multiplying the equations for \(u_1\) and \(u_2\) by \(\psi\), integrating by parts over \(Q^+\), and subtracting, we obtain
    \begin{equation*}
        \int_\Omega (u_1(x, T^*) - u_2(x, T^*))\,v(x, T^*)\,dx
        = \int_{\Sigma^+} (u_1 - u_2)\,\omega_v\,d\sigma.
    \end{equation*}
    Therefore,
    \begin{align*}
        \big|\mathcal{R}_1(v) - \mathcal{R}_2(v)\big|
        &\le \left|\int_{\Sigma^+} (u_1 - u_2)\,\omega_v\,d\sigma\right|
            + \left|\int_{\Sigma^-} (u_1 - u_2)\,\frac{\partial v}{\partial n}\,d\sigma\right| \\
        &\le C(\Omega,q)\Big(\|v(T^*)\|_{H^1(\Omega)} 
                            + \|v\|_{L^2(0,T;H^1(\Omega))}\Big)\,
            \|u_1 - u_2\|_{L^2(\Sigma)},
    \end{align*}
    where we have used the above control estimate for \(\omega_v\) and standard trace inequalities for the term on \(\Sigma^-\). This proves~\eqref{bound of Rv in para}.

    Next, we estimate \(\mathcal{T}_c(\mu,\nu)\). By Kantorovich duality,
    \begin{equation}\label{maximal def}
        \mathcal{T}_c(\mu, \nu) 
        = \max_{(\varphi, \psi) \in \widetilde{\Phi}_c} J(\varphi, \psi).
    \end{equation}
    Let \((\varphi^*, \psi^*)\) be a maximizer in~\eqref{maximal def}.  
    Applying Lemma~\ref{para combine test functino lem} with \((\varphi^*,\psi^*)\), we obtain a test function \(v^*\in\mathcal{V}_t\) such that
    \begin{equation*}
        J(\varphi^*, \psi^*) \leq \mathcal{R}_1(v^*) - \mathcal{R}_2(v^*).
    \end{equation*}
    Since \(\mathcal{T}_c(\mu,\nu) = J(\varphi^*,\psi^*) \ge 0\), it follows that
    \[
    \mathcal{T}_c(\mu,\nu)
    \le \mathcal{R}_1(v^*) - \mathcal{R}_2(v^*)
    \le \big|\mathcal{R}_1(v^*) - \mathcal{R}_2(v^*)\big|.
    \]

    Finally, from the construction in Lemma~\ref{para combine test functino lem} and the bound~\eqref{bound of v H1 para}, together with standard energy estimates for the adjoint equation, we obtain a uniform control
    \[
    \|v^*(T^*)\|_{H^1(\Omega)} 
    + \|v^*\|_{L^2(0,T;H^1(\Omega))}
    \le C_4'(\eta_1,\eta_2,K,M,\Omega,T^*,T,c),
    \]
    for some constant \(C_4'\) depending only on the indicated parameters.  
    Substituting this bound into~\eqref{bound of Rv in para} and absorbing constants into
    \[
    C_5(\eta_1,\eta_2,K,M,\Omega,T^*,T,c)
    := C(\Omega,q)\,C_4'(\eta_1,\eta_2,K,M,\Omega,T^*,T,c),
    \]
    we obtain
    \[
    \big|\mathcal{R}_1(v^*) - \mathcal{R}_2(v^*)\big|
    \le C_5(\eta_1,\eta_2,K,M,\Omega,T^*,T,c)\,
        \|u_1 - u_2\|_{L^2(\Sigma)}.
    \]
    Combining this with the previous inequality for \(\mathcal{T}_c(\mu,\nu)\) yields~\eqref{final para result}, and the proof is complete.

\end{proof}

\section{Concluding remarks}

This work develops an optimal-transport framework for the stability analysis of inverse point-source problems for variable-coefficient elliptic and parabolic equations. The central observation is that suitably constructed adjoint solutions allow Kantorovich dual potentials to be realized as boundary linear functionals, thereby linking OT geometry directly to boundary measurements. In the elliptic setting, this is accomplished through CGO-based interpolation functions with quantitative control of correction terms. In the parabolic setting, the same spatial construction is combined with boundary controllability to transfer interior information to the boundary. This yields global stability bounds that depend only on boundary data misfits. A parallel argument also applies to point-mass initial data; see ~\ref{appendix:initial-data}.

Several limitations and directions for further research remain. First, our analysis focuses on finitely many point sources with positive amplitudes and, in the parabolic case, band-limited temporal profiles. Extending the framework to more general classes of measures (e.g., measures with bounded densities, sparse superpositions of localized blobs, or signed sources) is an interesting open problem. Second, the present work assumes full boundary measurements; if one could obtain \(C^k\)-bounds for CGO solutions vanishing on portions of the boundary, analogous OT stability estimates for partial boundary data should be within reach. Third, it would be natural to investigate extensions to more general models, such as equations with advection terms, coupled systems, or Helmholtz-type equations at fixed frequency, where CGO constructions and controllability results are also available. Finally, on the numerical side, it would be of interest to explore how the OT-based stability estimates derived here can inform the design and analysis of OT-type misfit functionals in reconstruction algorithms, and to what extent the theoretical constants and separation conditions can be reflected in practical regularization strategies.

% \section{Conclusion and discussion}
% In this work, we establish stability results for inverse problems associated with both time-independent and parabolic equations in terms of the OT metric. Instead of working directly with the definition of the transport cost, we employ the Kantorovich duality theorem to combine with the boundary measurements. The stability analysis also relies on the construction and estimates of CGO solutions. It is worth noting that if bounds in the $C^k$ norm can be obtained for correction terms of CGO solutions vanishing on parts of the boundary, then analogous stability results can also be derived for partial boundary data. In future work, we plan to investigate more complex and general models, such as PDEs with advection terms, as well as more general source terms.

\appendix

\section{Proof of Lemma~\ref{invertible matrix thm} and Lemma~\ref{lemma:control estimate}}
\label{appendix:proofs}
In this appendix, we present detailed proofs of the two lemmas stated earlier. We begin with the proof of Lemma~\ref{invertible matrix thm}, where the construction of appropriate test functions plays a central role. These functions will be built using CGO solutions, along with precise estimates of the associated correction terms' $H^p$ norm. Furthermore, we invoke results from~\cite{fursikov_controllability_1996} to establish the existence and regularity of the boundary controls required in our analysis.

\subsection{Proof of Lemma~\ref{invertible matrix thm}}\label{appendix:proof of lemma3}

We now turn to the proof of Lemma~\ref{invertible matrix thm}. As a preliminary step, we outline the construction of a test function $v \in \mathcal{V}$ based on CGO solutions. Let us define $w(x) = \sqrt{\kappa(x)}\, v(x)$. Then the condition $v \in \mathcal{V}$ is equivalent to the function $w$ satisfying the elliptic partial differential equation
\begin{equation}\label{CGO equations}
    (\Delta + \tilde{q})\, w(x) = 0 \quad \text{in } \Omega,
\end{equation}
where $\tilde{q}$ is a complex-valued potential derived from the coefficients of the original problem. Since $\kappa\in H^{p+2}(\Omega)$ and has positive lower bound on $\Omega$, using the composition theorem for Sobolev spaces in \cite{brezis_composition_2001}, we know that 
\begin{equation*}
    \Delta \sqrt{\kappa(x)},\frac{1}{\sqrt{\kappa(x)}},\frac{1}{\kappa(x)}\in H^p(\Omega). 
\end{equation*}
On the other hand, since $p>\frac{d}{2}$, we know that $H^{p}(\Omega)$ is a Banach algebra under pointwise multiplication, which implies that
\begin{equation*}
    \tilde{q}= \frac{q(x)}{\kappa(x)} + \frac{\Delta \sqrt{\kappa(x)}}{\sqrt{\kappa(x)}}\in H^{p}(\Omega).
\end{equation*}

The construction of $w$ relies on the CGO framework developed in~\cite{bal_inverse_2010}, from which we adopt the following key lemma:

\begin{lem}\label{lemma: test function construction}
For every complex vector $\rho$ satisfying $\rho \cdot \rho=0$, suppose:
\begin{equation}\label{vector condition}
    |\rho|\geq C_1(\Omega,p)(1+||\tilde{q}||_{H^p(\Omega)}),
\end{equation}

where $C_1(\Omega,p)$ is a constant only related to $\Omega$ and $p$. Then we can construct a test function $v\in \mathcal{V}$ such that $v(x)=\frac{w(x)}{\sqrt{\kappa(x)}}$, where $w(x)=\mathrm{e}^{\rho \cdot x}\Big(1+\psi_\rho(x)\Big)$ satisfies \eqref{CGO equations}. Besides, the correction term $\psi_{\rho}$ belongs to $C^{2}(\Omega)$ and its norm satisfies:
\begin{equation}
    |\rho|\left\|\psi_\rho\right\|_{C^{1}(\Omega)}+\left\|\psi_\rho\right\|_{C^{2}(\Omega)} \leqslant C_1(\Omega,p) ||\tilde{q}||_{H^p(\Omega)}.
\end{equation}
\end{lem}
In addition, we derive several elementary estimates concerning the matrix $2$-norm, which will be essential in our construction.

\begin{lem}\label{lowerbound}
    Consider a general matrix $A = (a_{l,j})_{1 \leq l, j \leq m}$. Define
    \[
    \beta(A) := \min_{\|y\|_2 = 1} \|A y\|_2, \quad y \in \mathbb{C}^m.
    \]
    Then the following estimate holds:
    \begin{equation*}
       \beta(A) \geq \min_{1 \leq j \leq m} |a_{j,j}| - \Bigg( \sum_{1 \leq j \neq l \leq m} |a_{l,j}|^2 \Bigg)^{1/2}.
    \end{equation*}
    Moreover, for any matrix $B \in \mathbb{C}^{m \times m}$, we have
    \begin{equation*}
        \beta(A + B) \geq \beta(A) - \|B\|_2, \qquad \beta(AB) \geq \beta(A)\, \beta(B).
    \end{equation*}
    In particular, if $\beta(A) > 0$, then $A$ is invertible and satisfies
    \begin{equation*}
        \|A^{-1}\|_2 \leq \frac{1}{\beta(A)}.
    \end{equation*}
\end{lem}

\begin{proof}
    First, for any $y \in \mathbb{C}^{m}$ with $\|y\|_2 = 1$, we have
\[
\|(A+B)y\|_2 
    \;\ge\; \|Ay\|_2 - \|By\|_2
    \;\ge\; \beta(A)\|y\|_2 - \|B\|_2\|y\|_2
    \;=\; \beta(A) - \|B\|_2.
\]
Similarly,
\[
\|ABy\|_2 
    \;\ge\; \beta(A)\|By\|_2 
    \;\ge\; \beta(A)\,\beta(B)\|y\|_2
    \;=\; \beta(A)\,\beta(B).
\]

Let $D_A := \operatorname{diag}(a_{1,1}, \dots, a_{m,m})$.  
Because $D_A$ is diagonal, it follows immediately that
\[
\beta(D_A) \;=\; \min_{1 \le j \le m} |a_{j,j}|.
\]
Therefore,
\[
\beta(A)
    \;\ge\; \beta(D_A) - \|A - D_A\|_2
    \;\ge\; \min_{1 \le j \le m} |a_{j,j}|
           - \Bigg(\sum_{1 \le j \neq l \le m} |a_{l,j}|^{2}\Bigg)^{1/2}.
\]

Finally, since $\|Ay\|_2 \ge \beta(A)\|y\|_2 > 0$ for every non-zero $y$, the matrix $A$ is invertible. Moreover,
\[
1 \;=\; \|y\|_2 
    \;=\; \|A A^{-1} y\|_2 
    \;\ge\; \beta(A)\, \|A^{-1}y\|_2,
\]
and hence
\[
\beta(A)\,\|A^{-1}\|_2 \;\le\; 1.
\]

\end{proof}
We will use the following lemma to construct the complex vector:
\begin{lem}\label{lemma:vectorb}
   For any real vector $a\in \mathbb{R}^d$, there exists another real vector $b$, such that the complex vector $\rho =a+b\mathrm{i} \in \mathbb{C}^d$ satisfying $\rho \cdot \rho=0$.
\end{lem}
We now present the proof of Lemma~\ref{invertible matrix thm}.

\begin{proof}
For each $s_j$, we choose $\rho_j = \tilde{r} s_j + \mathrm{i} b_j$, where $b_j$ is selected according to Lemma~\ref{lemma:vectorb}. We define
\begin{equation*}
    \tilde{v}_j(x) = \frac{w_j(x)}{\sqrt{\kappa(x)\kappa(s_j)}}, \quad \text{where } w_j(x) = e^{\rho_j \cdot x}\left(1 + \psi_j(x)\right),
\end{equation*}
and $w_j$ is the CGO solution satisfying equation~\eqref{CGO equations}.

Define the diagonal matrix
\[
P = \operatorname{diag}\left( \exp\left(\frac{\tilde{r} s_j \cdot s_j}{2} \right) \frac{1}{\sqrt{\kappa(s_j)}} \right)_{1 \leq j \leq n}.
\]
Then we can write
\[
A = P (B + \delta B) P,
\]
where
\begin{align*}
    B &= \left[ \exp(\mathrm{i} b_l \cdot s_j) \exp\left(-\frac{\tilde{r} |s_j - s_l|^2}{2}\right) \right]_{1 \leq l,j \leq n}, \\
    \delta B &= \left[ \exp(\mathrm{i} b_l \cdot s_j) \exp\left(-\frac{\tilde{r} |s_j - s_l|^2}{2}\right) \psi_l(s_j) \right]_{1 \leq l,j \leq n}.
\end{align*}

We now apply Lemma~\ref{lowerbound} to estimate $\beta(B)$:
\begin{align*}
    \beta(B) 
    &\geq \min_{1 \leq j \leq n} |\exp(\mathrm{i} b_j \cdot s_j)| 
        - \left( \sum_{1 \leq j \neq l \leq n} 
        \left| \exp\left(\mathrm{i} b_l \cdot s_j\right) 
        \exp\left(-\frac{\tilde{r} |s_j - s_l|^2}{2} \right) \right|^2 \right)^{1/2} \\
    &= 1 - \left( \sum_{1 \leq j \neq l \leq n} 
        \exp\left(-\tilde{r} |s_j - s_l|^2 \right) \right)^{1/2} \\
    &\geq 1 - \sqrt{n(n-1)} \exp\left(-\frac{\tilde{r} \eta_1^2}{2} \right),
\end{align*}
where we used the uniform bound $|s_j - s_l| \geq \eta_1$ for $j \neq l$ in the last inequality. Therefore,
\begin{equation}\label{eq:betaB}
    \beta(B) \geq 1 - n \exp\left(-\frac{\tilde{r} \eta_1^2}{2} \right).
\end{equation}

Next, we estimate $\|\delta B\|_2$. Using the definition of $\delta B$ and bounding the exponential terms by $1$, we have
\begin{align*}
    \|\delta B\|_2 
    &\leq \left( \sum_{l,j=1}^n 
        \left| \exp(\mathrm{i} b_l \cdot s_j) 
        \exp\left(-\frac{\tilde{r} |s_j - s_l|^2}{2} \right) \psi_l(s_j) \right|^2 \right)^{1/2} \\
    &\leq \left( \sum_{l,j=1}^n 
        |\psi_l(s_j)|^2 \right)^{1/2}.
\end{align*}

By Lemma~\ref{lemma: test function construction}, we have the estimate
\[
|\rho_l| = \sqrt{2} \tilde{r} |s_l| \geq \sqrt{2} \tilde{r} \eta_2 
    \geq C_1\left(1 + \|\tilde{q}\|_{H^p(\Omega)} \right),
\]
which ensures that
\[
\|\psi_l\|_{C^0(\Omega)} \leq \frac{C_1 \|\tilde{q}\|_{H^p(\Omega)}}{\sqrt{2} \tilde{r} \eta_2}.
\]
Thus, we conclude that
\begin{equation}\label{eq:deltaB}
    \|\delta B\|_2 \leq n \cdot \frac{C_1 \|\tilde{q}\|_{H^p(\Omega)}}{\sqrt{2} \tilde{r} \eta_2}.
\end{equation}

Combining \eqref{eq:betaB} and \eqref{eq:deltaB}, we apply Lemma~\ref{lowerbound} to obtain
\begin{equation*}
    \beta(B + \delta B) \geq \beta(B) - \|\delta B\|_2 
    \geq 1 - n\left( \exp\left(-\frac{\tilde{r} \eta_1^2}{2} \right) + \frac{C_1 \|\tilde{q}\|_{H^p(\Omega)}}{\sqrt{2} \tilde{r} \eta_2} \right).
\end{equation*}

Since $P$ is diagonal, we have
\[
\beta(P) = \min_{1 \leq j \leq n} 
\left( \exp\left(\frac{\tilde{r} s_j \cdot s_j}{2} \right) 
\frac{1}{\sqrt{\kappa(s_j)}} \right)
\geq \frac{\exp\left(\frac{\tilde{r} \eta_2^2}{2} \right)}{\sqrt{ \|\kappa\|_{C^0(\Omega)} }}.
\]
Therefore, combining the estimates, we obtain
\begin{align*}
    \beta(A) 
    &\geq \beta(P)\, \beta(B + \delta B)\, \beta(P) \\
    &\geq \frac{\exp(\tilde{r} \eta_2^2)}{\|\kappa\|_{C^0(\Omega)}} 
    \left[ 1 - n \left( \exp\left(-\frac{\tilde{r} \eta_1^2}{2} \right) + \frac{C_1 \|\tilde{q}\|_{H^p(\Omega)}}{\sqrt{2} \tilde{r} \eta_2} \right) \right].
\end{align*}

By the definition of $\tilde{r}$, we have
\[
1 > n \left( \exp\left(-\frac{\tilde{r} \eta_1^2}{2} \right) + \frac{C_1 \|\tilde{q}\|_{H^p(\Omega)}}{\sqrt{2} \tilde{r} \eta_2} \right),
\]
which implies $\beta(A) > 0$, and hence $A$ is invertible by Lemma~\ref{lowerbound}.

It remains to estimate the $H^1$-norm of the test functions $\tilde{v}_j$. Since $\rho_j \cdot \rho_j = 0$, we observe that
\[
\tilde{r} |s_j| = |b_j| = \frac{|\rho_j|}{\sqrt{2}}.
\]
We estimate
\begin{align*}
    \|\tilde{v}_j\|_{C^1(\Omega)} 
    &\leq 3 \left\| \frac{1}{\sqrt{\kappa(x)}} \right\|_{C^1(\Omega)} 
    \left\| e^{\rho_j \cdot x} \right\|_{C^1(\Omega)} 
    \left( 1 + \|\psi_j\|_{C^1(\Omega)} \right) \\
    &\leq 3 \sqrt{2} \left\| \frac{1}{\sqrt{\kappa(x)}} \right\|_{C^1(\Omega)} 
    \tilde{r} |s_j|\, e^{\tilde{r} |s_j| R_0} 
    \left( 1 + \|\psi_j\|_{C^1(\Omega)} \right) \\
    &\leq 3 \sqrt{2} \tilde{r} |s_j| e^{\tilde{r} |s_j| R_0}\left\| \frac{1}{\sqrt{\kappa(x)}} \right\|_{C^1(\Omega)}
    \left( 1 + \frac{C_1 \|\tilde{q}\|_{H^p(\Omega)}}{\sqrt{2} \tilde{r} |s_j|} \right) \\
    &\leq 3 \sqrt{2} \tilde{r} R_0 e^{\tilde{r} R_0^2}\left\| \frac{1}{\sqrt{\kappa(x)}} \right\|_{C^1(\Omega)}
    \left( 1 + \frac{C_1 \|\tilde{q}\|_{H^p(\Omega)}}{\sqrt{2} \tilde{r} \eta_2} \right).
\end{align*}

Finally, since $\Omega$ is bounded, the $C^1$-norm controls the $H^1$-norm:
\[
\|\tilde{v}_j\|_{H^1(\Omega)} \leq \sqrt{(d + 1) |\Omega|} \cdot \|\tilde{v}_j\|_{C^1(\Omega)}.
\]
This completes the estimate for $\max_{1 \leq j \leq n} \|\tilde{v}_j\|_{H^1(\Omega)}$ and thus concludes the proof.
\end{proof}
\subsection{Proof of boundary control lemma}\label{appendix:proof of lemma4}
In this part, we give the detailed proof of Lemma \ref{lemma:control estimate}:
\begin{proof}
We first extend $\Omega$ to a larger bounded domain $\Omega_1$ and select an open subset $\Omega_0 \subset \Omega_1$ such that $\Omega_0 \cap \Omega = \emptyset$. By Theorem~5 in Chapter~3 of \cite{stein_singular_1970}, the coefficients $\kappa$ and $q$ can be extended to functions in $C^2(\Omega_1)$. Moreover, for any $v_0 \in H^1(\Omega)$, there exists an extension $v_1 \in H^1(\Omega_1)$ satisfying
\[
v_1|_{\Omega} = v_0, \quad \text{and} \quad \|v_1\|_{H^1(\Omega_1)} \leq C \|v_0\|_{H^1(\Omega)}.
\]

According to Theorem~2.1 in Chapter~1 of \cite{fursikov_controllability_1996}, for any $v_1 \in H^1(\Omega_1)$, there exists an internal control $f \in L^2\big([T^*, T] \times \Omega_0\big)$ and a corresponding solution $y$ to the following system:
\begin{align*}
\left\{
\begin{array}{rcll}
- \partial_t y + \nabla \cdot (\kappa \nabla y) - q y &=& f & \text{in } [T^*, T] \times \Omega_1, \\
\frac{\partial y}{\partial n} &=& 0 & \text{on } [T^*, T] \times \partial\Omega_1, \\
y(x, T) &=& 0 & \text{in } \Omega_1, \\
y(x, T^*) &=& v_1 & \text{in } \Omega_1.
\end{array}
\right.
\end{align*}
Furthermore, the solution $y$ satisfies the estimate:
\[
\|y\|_{L^2\big(T^*, T; H^2(\Omega_1)\big)} \leq C(\kappa, q) \|v_1\|_{H^1(\Omega_1)}.
\]

We now define the time-reversed function $\psi_1(x, t) := y(x, T - t + T^*)$, which satisfies:
\begin{align*}
\left\{
\begin{array}{rcll}
\partial_t \psi_1 + \nabla \cdot (\kappa \nabla \psi_1) - q \psi_1 &=& 0 & \text{in } [T^*, T] \times \Omega, \\
\frac{\partial \psi_1}{\partial n} &=& \omega_1 & \text{on } [T^*, T] \times \partial\Omega, \\
\psi_1(x, T^*) &=& 0 & \text{in } \Omega, \\
\psi_1(x, T) &=& v_0 & \text{in } \Omega,
\end{array}
\right.
\end{align*}
where $\omega_1 := \left. \frac{\partial y}{\partial n} \right|_{(x, T - t + T^*)}$.

By the continuity of the trace operator from $H^2(\Omega)$ to $H^{1/2}(\partial\Omega)$, we obtain the following estimate:
\begin{align*}
\|\omega_1\|_{L^2\big(T^*, T; H^{1/2}(\partial\Omega)\big)}
&= \left\| \frac{\partial y(x, T - t + T^*)}{\partial n} \right\|_{L^2\big(T^*, T; H^{1/2}(\partial\Omega)\big)} \\
&\leq C(\Omega) \|y\|_{L^2\big(T^*, T; H^2(\Omega)\big)} \\
&\leq C(\kappa, q, \Omega) \|v_0\|_{H^1(\Omega)}.
\end{align*}

We now choose $v_0 = -v(x, T)$ and define $\psi := \psi_1 + v(x, t)$. Then $\psi$ satisfies the target system \eqref{control system}. Moreover, the corresponding boundary control $\omega := \omega_1 + \frac{\partial v}{\partial n}$ satisfies the estimate \eqref{control estimate}, completing the proof.
\end{proof}

\section{OT stability for parabolic equations with point-mass initial data}
\label{appendix:initial-data}

This appendix records a companion result to the space-time source analysis 
developed in Sections~3-4. We show that the same CGO-based interpolation 
idea and duality mechanism yield an OT-type stability estimate for parabolic 
equations with point-mass initial data. The argument is a straightforward 
adaptation of the parabolic case, with the time dependence collapsing to a 
single initial-time layer. The result is included for completeness and to 
highlight the generality of the approach.

We begin with the following forward parabolic model:
\begin{align}\label{directmodel 2}
\left\{
\begin{array}{ccll}
  u_t - \Delta u + q(x)\,u &=& 0
  & \text{in } Q , \\[0.4em]
  \dfrac{\partial u}{\partial n} &=& 0 
  & \text{on } \Sigma ,\\[0.4em]
  u(x,0) &=& \displaystyle\sum_{j=1}^{m} a_j\,\delta(x-s_j)
  & \text{in } \Omega.
\end{array}
\right.
\end{align}
The available data consist of the boundary observations \(u|_{\Sigma}\).
Let \(\mu\) and \(\nu\) be two initial measures in \(\Theta_M\). 
As in the elliptic case, our goal is to derive a stability estimate of the form
\[
    \mathcal{T}_c(\mu,\nu)\leq C\|u_1-u_2\|_{L^2(\Sigma)},
\]
where \(u_1\) and \(u_2\) denote the corresponding boundary measurements.

To this end, we introduce a test function satisfying the adjoint equation
\[
    \partial_t v_1+\Delta v_1 - q v_1=0.
\]
Multiplying \eqref{directmodel 2} by \(v_1\) and integrating by parts over \(Q\), we obtain
\begin{equation}
    \int_{\Omega} \bigl[v_1(x,T)u(x,T)-v_1(x,0)u(x,0)\bigr]\,dx
    +\int_{\Sigma}u\,\frac{\partial v_1}{\partial n}\,ds=0.
\end{equation}

For any \(v\in \mathcal{V}\), the null controllability result of Lemma~\ref{lemma:control estimate} 
guarantees the existence of a boundary control \(\omega_v\) such that the following backward 
parabolic problem admits a solution:
\begin{align*}
\left\{
\begin{array}{rcll}
\partial_t v_1 + \Delta v_1 - q v_1 &=& 0 & \text{in } Q, \\
\dfrac{\partial v_1}{\partial n} &=& \omega_v & \text{on } \Sigma, \\
v_1(x, T) &=& 0 & \text{in } \Omega, \\
v_1(x, 0) &=& v & \text{in } \Omega.
\end{array}
\right.
\end{align*}
We then define the linear functional
\begin{equation}
    \mathcal{R}(v)=\int_{\Sigma}u\,\omega_v\,ds.
\end{equation}

\begin{lem}
For every \(v\in \mathcal{V}\), the following identity holds:
\begin{equation}
    \mathcal{R}(v)=\int_{\Omega} v\,d\mu
    =\sum_{j=1}^m a_j\,v(s_j).
\end{equation}
\end{lem}

We next select \(v\) from the family of basis functions constructed in 
Theorem~\ref{basis thm} and invoke the boundary control estimates for \(\omega_v\) 
provided by Lemma~\ref{lemma:control estimate}. 
By following the same argument as in Section~\ref{Main for elliptic}, 
we arrive at the desired stability result.

\begin{thm}
Let \(\mu\) and \(\nu\) be two probability measures in \(\Theta_M\). 
Then there exists a test function \(v^*\in\mathcal{V}\) such that
\begin{equation}\label{main result1}
    \mathcal{T}_c(\mu, \nu) 
    \leq \bigl|\mathcal{R}_1(v^*)-\mathcal{R}_2(v^*)\bigr|
    \leq C(\Omega,d,T)\,\|c\|_{\infty}\,
        M\,\tilde{r}\, e^{\tilde{r}(R_0^2 - \eta_2^2)}\,
        \|u_1 - u_2\|_{L^2(\Sigma)},
\end{equation}
where
\[
\tilde{r} = 4M \left( \frac{2}{\eta_1^2} 
      + \frac{1 + C_1 \|q\|_{H^p(\Omega)}}{\sqrt{2} \eta_2} \right),
\]
and \(C_1\) is as in Lemma~\ref{invertible matrix thm}.
\end{thm}

\bibliographystyle{elsarticle-num} 
\bibliography{ref}

\end{document}